\documentclass[final,3p,times]{elsarticle}

\usepackage{amssymb}

\usepackage{amssymb}
\usepackage{mathrsfs}
\usepackage{amsmath}
\usepackage{amsthm,amsfonts}
\usepackage{graphicx}
\usepackage{float}
\usepackage{color}
\usepackage{enumerate}
\usepackage[linesnumbered,ruled,vlined,boxed]{algorithm2e}
\usepackage{comment}
\usepackage{multirow}
\usepackage{slashbox}
\usepackage{CJK}

\usepackage{array}

\newtheorem{lemma}{Lemma}
\newtheorem{corollary}[lemma]{Corollary}
\newtheorem{theorem}[lemma]{Theorem}
\newtheorem{definition}[]{Definition}
\newtheorem{problem}[]{Problem}

\begin{document}

\begin{frontmatter}



\title{Metric dimension and edge metric dimension of unicyclic graphs }


\author[label1]{Enqiang Zhu}

\author[label1]{Shaoxiang Peng}

\author[label2]{Chanjuan Liu\footnote{corresponding author:chanjuanliu@dlut.edu.cn}}

\address[label1]{Institute of Computing Science and Technology, Guangzhou University, Guangzhou 510006, China}

\address[label2]{School of Computer Science and Technology, Dalian University of Technology, Dalian 116024, China}

\begin{abstract}
The metric (resp. edge metric) dimension of a simple connected graph $G$, denoted by dim$(G)$ (resp. edim$(G)$), is the cardinality of a smallest vertex subset $S\subseteq V(G)$ for which every two distinct vertices (resp.  edges) in $G$ have distinct distances to a vertex of $S$.
It is an interesting topic to discuss the relation between  dim$(G)$ and edim$(G)$ for some class of graphs $G$.
In this paper, we settle two open problems on this topic for a widely studied class of graphs, called unicyclic graphs. Specifically,  we introduce four classes of subgraphs to characterize the structure of a unicyclic graph whose metric (resp . edge metric) dimension is  equal to the lower bound on this invariant for unicyclic graphs. Based on this,
we  determine the exact values of dim$(G)$ and edim$(G)$ for all unicyclic graphs $G$.

\end{abstract}





\end{frontmatter}



\section{Introduction}
Graphs considered here should be finite, simple, and undirected.  We use $V(G)$ and $E(G)$ to denote the \emph{vertex set} and \emph{edge set} of a graph $G$, respectively.  The \emph{degree} of a vertex $v\in V(G)$ in $G$, denoted by $d_G(v)$, is the number of vertices adjacent to $v$. If a vertex has degree $k$ (resp. at least $k$) in $G$, then we call it a $k$-vertex (resp. $k^{+}$-vertex) of $G$.  
The \emph{distance} between  two vertices $u,v\in V(G)$, denoted by $d_G(u,v)$, is the length of a shortest path from $u$ to $v$, and the \emph{distance} between  a vertex $u\in V(G)$ and an edge $e\in E(G)$, denoted by $d_G(u, e)$, is defined as $d_G(u, e)=\min \{d_G(u,v_1), d_G(u,v_2)\}$, where $e=v_1v_2$.  A \emph{path} from a vertex $u$ to an edge $e$ refers to a path from $u$ to an arbitrary endpoint of $e$. Throughout this paper, the notation $P_{u,v}$ (resp. $P_{u,e}$) is used to denote the shortest path from vertex $u$ to vertex $v$ (resp. vertex $u$ to edge $e$).
Let $G$ be a connected graph and $S\subseteq V(G)$. For any $u,u'\in V(G)$ (resp.  $e,e' \in E(G)$),  we say that $S$ \emph{distinguishes} $u$ and $u'$ (resp. $e$ and $e'$)  if there exists a vertex $v\in S$ such that  $d_G(v,u)\neq d_G(v,u')$ (resp. $d_G(v,e)\neq d_G(v,e')$); it is also said that $v$ distinguishes  $u,u'$ (resp. $e,e'$).
If every pair of vertices (resp. edges) of $G$ can be distinguished by  $S$, then we call $S$  a \emph{metric generator} (resp. \emph{an edge metric generator}) of $G$. The minimum $k$ for which $G$ has a metric (resp. edge metric) generator of cardinality $k$ is called the \emph{metric} (resp.\emph{edge metric}) \emph{dimension} of $G$, denoted by dim($G$) (resp. edim($G$)).

One of the first people to define metric dimension was Slater \cite{slater1975leaves}, who used the term `location number' originally in connection with a location problem of uniquely determining the position of an intruder in a network.   This concept was  independently discovered by Harary and Melter \cite{Harary1976}, who was the first to use the term `metric dimension'.  Metric dimension is an important concept in graph theory,
and has many applications in diverse areas, such as  image processing and pattern recognition \cite{melter1984metric}, robot navigation \cite{khuller1996landmarks}, and image processing and pattern recognition \cite{caceres2007metric},  to name a few. It has been shown that determining metric dimension is \textbf{NP}-complete for planar graphs of maximum degree 6 \cite{diaz2017complexity} and split graphs (or interval graphs, or permutation graphs) of diameter 2 \cite{foucaud2017identification}.
Recently, motivated by an observation that there exist graphs where any metric generator does not distinguish all pairs of edges, Kelenc, et al. \cite{kelenc2018uniquely} proposed the concept of  \emph{edge metric generator} which aims to distinguish edges instead of vertices, and made a comparison between these two dimensions.  They also proved that determining the
edge metric dimension  is \textbf{NP}-hard, and proposed some open problems, most of which have been settled subsequently \cite{zubrilina2018edge, zhu2019graphs}. For more information on edge metric dimension, please refer to \cite{filipovic2019edge,peterin2020edge,zhang2020edge}.

By the similarity of the metric dimension and the edge metric dimension, it is interesting to study the relation between them. In \cite{kelenc2018uniquely}, it was proved that there exist several families of graphs $G$ satisfying dim($G$)$<$ edim($G$), dim($G$)$=$ edim($G$), or dim($G$)$>$ edim($G$). In \cite{zubrilina2018edge}, it was proved that $\frac{\mathrm{dim}(G)}{\mathrm{edim}(G)}$ is not bounded from above. More recently, Knor et al. \cite{knor2021graphs}
further studied the ratio $\frac{\mathrm{dim}(G)}{\mathrm{edim}(G)}$, and by constructing  graphs based on a family of special unicyclic graphs they proved that there are graphs $G$ for which both $\mathrm{dim}(G)-\mathrm{edim}(G)$ and $\mathrm{edim}(G)-\mathrm{dim}(G)$ can be arbitrary large. They also posed two problems on how to characterize unicyclic graphs $G$ with $\mathrm{edim}(G)<\mathrm{dim}(G)$ or $\mathrm{edim}(G)=\mathrm{dim}(G)-1$.

\begin{problem}\cite{knor2021graphs}\label{pro1}
Characterize the class of unicyclic graphs $G$ for which $\mathrm{edim}(G)<\mathrm{dim}(G)$ (resp. $\mathrm{edim}(G)=\mathrm{dim}(G)-1$.
\end{problem}
\noindent Sedlar and \v{S}krekovski \cite{sedlar2021} gave the bounds on $\mathrm{dim}(G)$ and $\mathrm{edim}(G)$ for unicyclic graphs $G$, and proved that $|\mathrm{dim}(G)-\mathrm{edim}(G)|\leq 1$. Furthermore, they proposed a similar problem as Problem \ref{pro1}.

\begin{problem}\cite{sedlar2021} \label{pro2}
For a unicyclic graph $G$, determine when the difference  $\mathrm{dim}(G)-\mathrm{edim}(G)=1,0,$ or $-1$.
\end{problem}

In this paper, by introducing four classes of subgraphs, we determine the exact values of dim$(G)$ and edim$(G)$ for unicyclic graphs $G$. Our results settle the above two problems. The section below describes terminologies and notations we will adopt in our proofs and some (known and new) conclusions.

\section{Preliminaries}
We follow the same notations in \cite{sedlar2021} with a slight change. Given a graph $G$, we use $G-V'$ and $G-E'$ to denote the resulting graph obtained from $G$ by deleting all vertices in $V'$ (and all of theirs incident edges) and all edges in $E'$, respectively, where $V'\subseteq V$ and $E'\subseteq E(G)$. Especially, when $V'=\{v\}$ and $E'=\{e\}$, we replace them by $G-v$ and $G-e$, respectively. For two integers $i, j$ such that $0\leq i<j$, we use $[i,j]$ to denote the set $\{i, i+1, \ldots, j\}$.
The focus of this paper is on unicyclic graphs, which are graphs  containing only one cycle. Let $G$ be a unicyclic graph.  For the sake of convenience, we use $\mathcal{C}(G)$ and $\ell(G)$ (or simply $\mathcal{C}$ and $\ell$ when $G$ is clear from the context) to denote the unique cycle of $G$ and the  length of $\mathcal{C}$, respectively, and denote by $\mathcal{O}(\ell)$ and $\mathcal{E}(\ell)$  the sets of unicyclic graphs with $\ell  \equiv 1$ (mod 2) and $\ell \equiv 0$ (mod 2), respectively. Clearly, $\ell\geq 3$.

In what follows, for any unicyclic graph $G$, we  always let $\mathcal{C}=a_0a_1\ldots a_{\ell-1}a_0$.
Note that $G-E(\mathcal{C})$ is a forest. For each $i\in [0,\ell-1]$, we denote by $T_{a_i}$ the component of $G-E(\mathcal{C})$ containing $a_i$, and call $a_i$ the \emph{master} of every vertex in $V(T_{a_i})$ (it is necessary here to clarify that $a_i$ is also its own master).
For any $S\subseteq V(G)$, we say that $a_i, i\in [0,\ell-1]$ is \emph{$S$-active} if $T_{a_i}$ contains a vertex of $S$, and use $\mathcal{A}(S)$ to denote the set of all $S$-active vertices.
A vertex $v\in V(G)$ is called a \emph{branching vertex} if either $v\in V(\mathcal{C})$ and $d_G(v)\geq 4$ or $v\notin V(\mathcal{C})$ and $d_G(v)\geq 3$. Clearly, if $T_{a_i}$ contains no branching vertex, then $T_{a_i}$ is a path.
For every vertex $v$ with $d_G(v)\geq 3$, a \emph{thread} attached to $v$ is a path $v_1\ldots v_k$ such that $v_i$ is a 2-vertex for $i\in[1, k-1]$, $v_k$ is an 1-vertex, and $v_1v\in E(G)$, where $k\geq 1$.
Let us denote by $\mathcal{S}(v)$ the set of all threads attached to $v$. A $4^+$-vertex $v \in V(\mathcal{C})$ is called a \emph{bad} vertex of $G$ if there is a thread attached to $v$. A \emph{branch-resolving set} of  $G$ is a set $S\subseteq V(G)$ such that for every $3^+$-vertex  $v$,  at least $|\mathcal{S}(v)|-1$ threads in $\mathcal{S}(v)$ contain a vertex of $S$. Indeed, we are interested in the branch-resolving set containing the minimum number of vertices. Obviously, the cardinality of any minimum branch-resolving set, denoted by $L(G)$, is determined by
\begin{equation*}
 L(G)= \sum\limits_{v\in V(G), |\mathcal{S}(v)|>1} (|\mathcal{S}(v)|-1).
\end{equation*}
Observe that each thread has exactly one 1-vertex; therefore, we can choose a minimum branch-resolving set which consists of only 1-vertices. We use $\mathcal{B}(G)$ to denote the set of all such  minimum branch-resolving sets of $G$. When $S (\in \mathcal{B}(G)) \neq \emptyset$, we can label vertices  $a_0,a_1,\ldots, a_{\ell-1}$  in a specific way: $a_0 \in \mathcal{A}(S)$ and the maximum $i$ such that $a_i$ is an $S$-active vertex should be as small as possible. Such a labeling is called a \emph{normal labeling with respect to $S$} of $\mathcal{C}$ (or simply a \emph{normal labeling} of $\mathcal{C}$).

Let $G$ be a unicyclic graph and $a_i,a_j,a_k$ be three vertices on $\mathcal{C}$. If $d_G(a_i,a_j)+d_G(a_j,a_k)+d_G(a_k,a_i)= \ell$, then $a_i,a_j,a_k$ are said to form a geodesic triple. Let $S\in \mathcal{B}(G)$ and $r = \max \{ i | a_i$ is $S$-active\}.
Notice that when $\mathcal{C}$ is labelled normally, if $r>\lfloor\frac{\ell}{2}\rfloor$, then there must exist three $S$-active vertices forming  a geodesic triple; in addition, if $\ell\equiv 0$ (mod 2) and $r=\lfloor\frac{\ell}{2}\rfloor$, then either $\mathcal{A}(S) = \{a_0, a_{r}\}$ or there exist three $S$-active vertices forming  a geodesic triple.

The following part of this section moves on to describe some basic results,  which are helpful to prove our main conclusions  in the subsequent sections. We first present three known results given by Sedlar and \v{S}krekovski \cite{sedlar2021}.



\begin{lemma} \cite{sedlar2021} \label{necessary-0}
Let $G$ be a unicyclic graph. Then, every metric generator (resp. edge metric generator) $S$ is a branch-resolving set  with $|\mathcal{A}(S)|\geq 2$.
\end{lemma}


\begin{lemma} \label{geodesic-triple}\cite{sedlar2021}
Let $G$ be a unicyclic graph. If $S\subseteq V(G)$ is a branch-resolving set of $G$ for which there are three $S$-active vertices forming a geodesic triple,  then $S$ is both a metric generator and an edge metric generator of $G$.
\end{lemma}

\begin{lemma} \cite{sedlar2021} \label{Sedlar-1}
Let $G$ be a unicyclic graph. If $S \subseteq V(G)$ is a branch-resolving set of $G$  such that  $|\mathcal{A}(S)|\geq 2$, then any two vertices (also any two edges) from each $T_{a_i}$, $i\in [0, \ell-1]$, can be distinguished by $S$.
\end{lemma}

What follows are three useful lemmas showing that $S\in \mathcal{B}(G)$ can distinguish some special pair of vertices and edges.

\begin{lemma} \label{add1}
Suppose that $G$ is a unicyclic graph with $|\mathcal{A}(S)|\geq 2$, where $S\in \mathcal{B}(G)$. Let $a_i, a_j$ be  two $S$-active vertices such that $0< |j-i| <\lfloor\frac{\ell}{2}\rfloor$ (relabel $\mathcal{C}$ if necessary). If  neither $a_i$ nor $a_j$ is a bad vertex, then any two vertices $v_1,v_2$ (resp. any two edges $e_1,e_2$) such that  $\{v_1,v_2\} \cap (V(T_{a_i}\cup V(T_{a_j})) \neq \emptyset$ (resp. $\{e_1,e_2\} \cap (E(T_{a_i}\cup E(T_{a_j})) \neq \emptyset$) can be distinguished by $S$.
\end{lemma}
\begin{proof}
By symmetry, let $i=0$ and $i<j<\lfloor\frac{\ell}{2}\rfloor$.
It suffices to deal with  $\{v_1,v_2\} \cap V(T_{a_0}) \neq \emptyset$ (resp. $\{e_1,e_2\} \cap E(T_{a_0}) \neq \emptyset$), since the case that  $\{v_1,v_2\} \cap V(T_{a_j}) \neq \emptyset$ (resp. $\{e_1,e_2\} \cap E(T_{a_j}) \neq \emptyset$) can be addressed in a similar way.
Without loss of generality,  assume that $v_1\in V(T_{a_0})$ (resp. $e_1\in E(T_{a_0})$) and let $v_2 \in V(T_{a_{j'}})$ (resp. $e_2 \in E(T_{a_{j'}})$ or $e_2=a_{j'}a_{j'+1}\in E(\mathcal{C})$), where $j'\in [0,\ell-1]$ and $a_{\ell}=a_0$. By Lemma \ref{Sedlar-1}, we only consider the case  $v_2 \notin V(T_{a_0})$ (resp. $e_2\notin E(T_{a_0})$). Let
$u\in (V_{T_{a_0}}\cap S)$ and $v\in  (V_{T_{a_0}}\cap S)$.
In the remainder of the proof, we suppose that $d_G(u, v_1)=d_G(u, v_2)$ (resp. $d_G(u, e_1)=d_G(u, e_2)$). Clearly, $v_1 \notin V(P_{u,a_0})$ (resp. $e_1 \notin E(P_{u,a_0}$), otherwise $d_G(u, v_1) < d_G(u,v_2)$ (resp. $d_G(u, e_1) < d_G(u,e_2)$).
Observe that $d_G(u)=1$; so, $u \notin V(P_{a_0, v_1})$ (resp. $u \notin E(P_{a_0,e_1}$).

Denote by $x \in V(P_{u, v_1}) \cap V(P_{u,a_0})$ (resp. $x \in V(P_{u, e_1}) \cap V(P_{u, a_0})$) the vertex that has the minimum distance with $a_0$ ($x=a_0$ is probable). Then, $d_G(a_0, u)= d_G(a_0, x)+ d_G(x, u)$, $d_G(a_0, v_1)= d_G(a_0, x)+ d_G(x, v_1)$ (resp. $d_G(a_0, e_1)=d_G(a_0, x)+ d_G(x, e_1)$), and
\begin{center}
$\left\{ \begin{array}{lc}
d_G(u, v_1)=d_G(u, x) + d_G(x, v_1)\\
d_G(u, v_2)=d_G(u, x)+d_G(x, a_0)+ d_G(a_0, v_2)\\
d_G(v,v_1)=d_G(v,a_j)+j+ d_G( x, a_0)+ d_G(x, v_1)\\
d_G(v,v_2)=d_G(v,a_j)+d_G(a_j,a_{j'})+ d_G(a_{j'}, v_2)
\end{array}\right.$
~{\Large{(}}resp.~
$\left\{ \begin{array}{lc}
d_G(u, e_1)=d_G(u, x) + d_G(x, e_1)\\
d_G(u, e_2)= d_G(u, x)+d_G(x, a_0)+ d_G(a_0, e_2)\\
d_G(v,e_1)= d_G(v,a_j)+j+ d_G(a_0, x)+ d_G(x, e_1)\\
d_G(v,e_2)=d_G(v,a_j)+d_G(a_j, e_2)
\end{array}\right.${\Large{)}}
\end{center}
By $d_G(u, v_1)= d_G(u, v_2)$ (resp. $d_G(u, e_1)= d_G(u, e_2)$), we have
$d_G(x,v_1) = d_G(x,a_0)+ d_G(a_{0},v_{j'})+d_G(a_{j'},v_2)$ (resp. $d_G(x, e_1) = d_G(x, a_0)+d_G(a_{0}, e_2)$). Now, suppose that $d_G(v,v_1)=d_G(v,v_2)$ (resp. $d_G(v,e_1)=d_G(v,e_2)$). Then,
\vspace{-0.05cm}
\begin{equation}\label{equ-1}
d_G(a_j,a_{j'})= j+2d_G(x,a_0)+d_G(a_0, a_{j'})~~~(resp.~~~
d_G(a_j,e_2)= j+ 2d_G(x, a_0)+d_G(a_{0}, e_2)~)
\end{equation}

\vspace{-0.05cm}
Case 1. $x\neq a_0$. Then, $d_G(x,a_0)>0$. By Equation (\ref{equ-1}), $j'>\lfloor\frac{\ell}{2}\rfloor$. So, $d_G(a_0, a_{j'}) =\ell-j'$ (resp. $d_G(a_0, e_2) =\ell-j'+ d_G(a_{j'}, e_2)$ when $e_2\notin E(\mathcal{C})$ and  $d_G(a_0, e_2) =\ell-j'-1$  when $e_2\in E(\mathcal{C})$).
If $j'-j \leq \lfloor\frac{\ell}{2}\rfloor$,  then  $d_G(a_j,a_{j'})= j'-j$ (resp. $d_G(a_j,e_2)= j'-j+ d_G(a_{j'}, e_2)$ when $e_2\notin E(\mathcal{C})$ and  $d_G(a_j,e_2)= j'-j$ when $e_2\in E(\mathcal{C})$), which implies (by Equation (\ref{equ-1})) that $2d_G(x,a_0) = 2j'-2j-\ell \leq 0$ (or $2d_G(x,a_0) = 2j'-2j-\ell+1 \leq 0$ when $e_2\in E(\mathcal{C})$, since $\ell \equiv 1$ (mod 2) in this case), a contradiction.  If $j'-j >\lfloor\frac{\ell}{2}\rfloor$, then $d_G(a_j,a_{j'})= \ell-j'+j$ (resp. $d_G(a_j,e_2)= \ell-j'+j+ d_G(a_{j'}, e_2)$ when $e_2\notin E(\mathcal{C})$ and  $d_G(a_j,e_2)= \ell-j'-1+j$ when $e_2\in E(\mathcal{C})$), which implies (by Equation (\ref{equ-1})) that $d_G(x,a_0)=0$,  a contradiction.

Case 2. $x=a_0$. Then, $d_G(u, v_1)= d_G(u, v_2)$ (resp. $d_G(u, e_1)= d_G(u, e_2)$) implies that $d_G(a_0,v_1)= d_G(a_0,v_2)$ (resp. $d_G(a_0,e_1)= d_G(a_0,e_2)$).
Notice that $v_1\neq a_0$ (resp. $e_1 \notin E(P_{u,a_0})$). We have that  $T_{a_0}-a_0$  is disconnected. Let $T'$ be the component of $T_{a_0}-a_0$ that contains  $v_1$ (resp. $e_1$ or one endpoint of $e_1$ when $e_1a_0\in E(G)$). Clearly, $u\notin V(T')$.
Since $a_0$ is not a bad vertex,  $V(T')\cap S \neq \emptyset$. Let $w\in (V(T')\cap S)$ and let $w' \in (V(P_{w, v_1}) \cap V(P_{w, a_0}))$ (resp. $w' \in (V(P_{w, e_1})\cap V(P_{a_0,v_1}))$) be the vertex that has the minimum distance with $a_0$. Clearly, $w'\neq a_0$. Observe that $d_G(w)=1$. So,  $w\notin V(P_{a_0, v_1})$ (except for the case that $w=v_1$, for which  $d_G(w, v_1)=0<d_G(w,v_2)$) and $w\notin V(P_{a_0, e_1})$. We also assume that $v_1$ (resp. $e_1$) is not on $P_{a_0,w}$; otherwise  $d_G(w,v_1)<d_G(w,v_2)$ (resp. $d_G(w,e_1)<d_G(w,e_2)$). Thus,
$d_G(w, v_1) = d_G(w,w')+ d_G(w', v_1)$ and $d_G(a_0, v_1)=d_G(a_0, w')+ d_G(w', v_1)$ (resp. $d_G(w, e_1) = d_G(w, w')+ d_G(w', e_1)$ and $d_G(a_0, e_1)=d_G(a_0, w')+ d_G(w', e_1)$). As a result,   $d_G(w, v_2) = d_G(w,w')+ d_G(w', a_{0})+ d(a_0, v_2)= d_G(w,w')+ d_G(w', a_{0})+ d_G(a_0, v_1)$ (resp. $d_G(w, e_2) =  d_G(w,w')+ d_G(w', a_{0})+ d_G(a_0, e_1)$). If $d_G(w, v_1) = d_G(w, v_2)$ (resp. $d_G(w, e_1) = d_G(w, e_2)$), then  $d_G(a_0, w')=0$, i.e., $a_0=w'$, a contradiction.
\end{proof}

\begin{lemma} \label{add2}
Let $G\in \mathcal{E}(\ell)$ and $S\in \mathcal{B}(G)$. Suppose that $\mathcal{A}(S)$ contains two vertices, say $a_0, a_j$, such that $1\leq j \leq  \frac{\ell-2}{2}$. Then, each vertex $w\in (V(T_{a_k})\cap S)$ for $k\in [1, j-1]$ can distinguish two vertices $v_1\in V(T_{a_{i'}})$ and $v_2\in V(T_{a_{j'}})$ (resp. two edges $e_1\in E(T_{a_{i'}})$ and $e_2\in V(T_{a_{j'}})$) for $i'\in [1, j-1]$ and $j'\in [\frac{\ell}{2}+1, \frac{\ell}{2}+j]$ such that $d_G(a_0,v_1)=d_G(a_0,v_2)$ and  $d_G(a_j,v_1)=d_G(a_j,v_2)$ (resp. $d_G(a_0,e_1)=d_G(a_0,e_2)$ and  $d_G(a_j,e_1)=d_G(a_j,e_2)$).
\end{lemma}

\begin{proof}
For convenience, we use $x_i$ to denote $v_i$ or $e_i$ for $i\in [1,2]$, and let $d_G(a_{i'}, v_i)=m_i$ (resp. $d_G(a_{i'}, e_i)=m_i$). Then, for any $x_i \in \{v_i,e_i\}$ ($r\in [1,2]$),
$d_G(a_0, x_1)=i'+m_1, d_G(a_0, x_2)=\ell-j'+m_2, d_G(a_j, x_1)=j-i'+m_1,$ and $d_G(a_j, x_2)=j'-j+m_2$.
By $d_G(a_0,v_1)=d_G(a_0,v_2)$ and  $d_G(a_j,v_1)=d_G(a_j,v_2)$ (resp. $d_G(a_0,e_1)=d_G(a_0,e_2)$ and  $d_G(a_j,e_1)=d_G(a_j,e_2)$), we deduce that $m_1-m_2=\ell-(i'+j')$ and $\ell=2(i'+j'-j)$.
Let $d_G(w, a_k)=m$.

Case 1.  $k\neq i'$. In this case,  $d_G(w, x_1)=m+|k-i'|+m_1$. When $j'-k\leq\frac{\ell}{2}$,  it has that $d_G(w, x_2)= m+j'-k+m_2$. If $d_G(w, x_1)=d_G(w, x_2)$, then  either $\ell=2(i'+j'-k)$ (when $k>i'$) which implies that $k=j$, or $j'= \frac{\ell}{2}$ (when $k<i'$). When $j'-k> \frac{\ell}{2}$,  it has that $d_G(w, x_2)= m+k+\ell-j'+m_2$. If $d_G(w, x_1)=d_G(w, x_2)$, then  either $i'=0$ (when $k>i'$), or $k=0$ (when $k<i'$). Clearly, each of these cases yields a contradiction.

Case 2. $k=i'$. If $x_1$  is on $P_{w, a_{i'}}$, it is clear that  $d_G(w, v_1) < d_G(w,v_2)$ (resp. $d_G(w, e_1) < d_G(w, e_2)$). So, suppose that $x_1$ is not on $P_{w, a_{i'}}$.  Let $w'$ be the common vertex shared by $P_{w, a_{i'}}$ and $P_{w, x_1}$ that has the minimum distance with $a_{i'}$. Then, $d_G(w, x_1)=d_G(w,w')+d_G(w', x_1)$ and $m_1= d_G(a_{i'}, w') + d_G(w', x_1)$. Note that $j'\geq \frac{\ell}{2}+1$, $m_2 = m_1-\ell+(i'+j')$, and $d_G(w', x_1) \leq  m_1$. When $j'-i'\leq \frac{\ell}{2}$, we have that $d_G(w, x_2)= d_G(w,w')+d_G(w',a_{i'})+j'-i'+m_2 = d_G(w,w')+d_G(w',a_{i'})+(2j'-\ell)+m_1 > d_G(w,w')+d_G(w', x_1) = d_G(w, x_1)$.  When $j'-i'> \frac{\ell}{2}$, we have that $d_G(w, x_2) = d_G(w,w')+d_G(w',a_{i'})+\ell-j'+i'+m_2 = d_G(w,w')+d_G(w',a_{i'})+2i'+m_1 > d_G(w,w')+d_G(w', x_1) = d_G(w, x_1)$.
\end{proof}

\begin{lemma}\label{add3}
Suppose that $G$ is a unicyclic graph. Let $a_i, a_j$ be any two  vertices on $\mathcal{C}$, say $a_0,a_j$ such that $0<j\leq \lfloor\frac{\ell-1}{2}\rfloor$. Then, any two vertices  $v_1 \in V(T_{a_{i'}})$ and $v_2\in V(T_{a_{j'}})$ (resp. two edges $e_1 \in E(T_{a_{i'}})$ and $e_2\in E(T_{a_{j'}}$) such that  $i',j'\in [0,\ell-1] \setminus \{i,j\}$ and $i'<j'$  can not be distinguished by $\{a_0,a_j\}$ iff one of the following conclusions holds, where $m_1=d_{G}(a_{i'}, v_1)$ and $m_2 = d_{G}(a_{i'}, v_2)$ (resp. $m_1=d_{G}(a_{i'}, e_1)$ and $m_2 = d_{G}(a_{i'}, e_2)$).

(i) $j'\in [j+1, \lfloor\frac{\ell}{2}\rfloor]$, $i'>j$, and $m_1-m_2=j'-i'$; or

(ii) $j'\in [\lfloor\frac{\ell}{2}\rfloor+1, \lfloor\frac{\ell}{2}\rfloor+j]$, $i'<j$, $\ell=2i'+2j'-2j$, and $m_1-m_2=\ell-j'-i'$; or

(iii) $j'\in [\lfloor\frac{\ell}{2}\rfloor+j+1, \ell-1]$, $\ell=2(i'-j)$, and $m_1-m_2=i'-j'$; or  $i'> \lfloor\frac{\ell}{2}\rfloor+j$ and $m_1-m_2=i'-j'$.
\end{lemma}

\begin{proof}
Notice that when (i) holds,  $d_G(a_0, v_1)=i'+m_1, d_G(a_0, v_2)=j'+m_2, d_G(a_j, v_1)=i'-j+m_1$, and  $d_G(a_j, v_2)=j'-j+m_2$; when (ii) holds, $d_G(a_0, v_1)=i'+m_1, d_G(a_0, v_2)=\ell-j'+m_2, d_G(a_j, v_1)=j-i'+m_1$, and  $d_G(a_j, v_2)=j'-j+m_2$; when (iii) holds, $d_G(a_0, v_1)=\ell-i'+m_1, d_G(a_0, v_2)=\ell-j'+m_2, d_G(a_j, v_1)=\frac{\ell}{2}+m_1$, and  $d_G(a_j, v_2)=\ell+j-j'+m_2$. One can readily check that $d_G(a_0, v_1)= d_G(a_0, v_2)$ and $d_G(a_j, v_1)= d_G(a_j, v_2)$ when one of the three conditions in the theorem holds.

Now, suppose that $v_1,v_2$ (resp. $e_1,e_2$) can not be distinguished by $S$, i.e., $d_G(a_0, a_{i'})+m_1= d_G(a_0, a_{j'})+m_2$ and $d_G(a_j, a_{i'})+m_1= d_G(a_j, a_{j'})+m_2$. Then,
\begin{equation}\label{addeq}
 d_G(a_0, a_{i'}) -d_G(a_0, a_{j'})= d_G(a_j, a_{i'})- d_G(a_j, a_{j'}),~~ \mathrm{and} ~~m_1-m_2=d_G(a_0,a_{j'})-d_G(a_0,a_{i'})
\end{equation}
Clearly, $j'>j$ (otherwise, $d_G(a_0, a_{i'}) -d_G(a_0, a_{j'})<0$ but $d_G(a_j, a_{i'})- d_G(a_j, a_{j'})>0$, a contradiction).

Case 1. $j'\in [j+1, \lfloor\frac{\ell}{2}\rfloor]$. Then, $d_G(a_0, a_{i'})=i', d_G(a_0, a_{j'})=j', d_G(a_j, a_{i'})=|j-i'|$, and $d_G(a_j, a_{j'})=j'-j.$ By Equation (\ref{addeq}), if $i'<j$, then we derive $i'=j$, a contradiction. Therefore, $i'>j$, and hence $m_1-m_2=j'-i'$.

Case 2. $j'\in [\lfloor\frac{\ell}{2}\rfloor+1, \lfloor\frac{\ell}{2}\rfloor+j]$. When $i'>\lfloor\frac{\ell}{2}\rfloor$, $d_G(a_0, a_{i'})=\ell-i', d_G(a_0, a_{j'})=\ell-j', d_G(a_j, a_{i'})=i'-j$, and $d_G(a_j, a_{j'})=j'-j.$ By Equation (\ref{addeq}), we derive $i'=j'$, a contradiction. When $i'\leq \lfloor\frac{\ell}{2}\rfloor$, $d_G(a_0, a_{i'})=i', d_G(a_0, a_{j'})=\ell-j', d_G(a_j, a_{i'})=|j-i'|$, and $d_G(a_j, a_{j'})=j'-j.$ By Equation (\ref{addeq}), if $i'>j$, then we deduce that $\ell=2j'$, a contradiction. Therefore, it holds that $i'<j$, by which we have $\ell=2i'+2j'-2j$ and $m_1-m_2=\ell-j'-i'$.

Case 3.  $j'\in [\lfloor\frac{\ell}{2}\rfloor+j+1, \ell-1]$. When $i'\leq \lfloor\frac{\ell}{2}\rfloor$, $d_G(a_0, a_{i'})=i', d_G(a_0, a_{j'})=\ell-j', d_G(a_j, a_{i'})=|j-i'|$, and $d_G(a_j, a_{j'})=\ell-j'+j.$ By Equation (\ref{addeq}),  we deduce that either $i'=0$ (when $i'<j$) or $j=0$ (when $i'>j$). Therefore, $i'>\lfloor\frac{\ell}{2}\rfloor$. Moreover, when  $\lfloor\frac{\ell}{2}\rfloor<i'\leq \lfloor\frac{\ell}{2}\rfloor+j$, $d_G(a_0, a_{i'})=\ell-i', d_G(a_0, a_{j'})=\ell-j', d_G(a_j, a_{i'})=i'-j$, and $d_G(a_j, a_{j'})=\ell-j'+j.$ By Equation (\ref{addeq}), we deduce that $\ell=2i'-2j$ and $m_1-m_2=i'-j'$.
When $i'> \lfloor\frac{\ell}{2}\rfloor+j$, then $d_G(a_0, a_{i'})=\ell-i'$ and $d_G(a_0, a_{j'})=\ell-j'$, which implies that $m_1-m_2=i'-j'$.
\end{proof}

\section{Metric Dimension} \label{sec-1}

This section is devoted to the argument of our main results for metric dimension. Similar conclusions for edge metric dimension will be presented in the next section. To achieve our purpose, we need to construct four families of graphs $G$ for which dim($G$)= $L(G)$ or edim($G$)= $L(G)$. The first two families of graphs (for metric dimension) are described here (Definition \ref{def-new1}), while the other two families of graphs (for edge metric dimension) will be introduced  in the following section (Definition \ref{def-new2}).

\begin{definition}\label{def-new1}
Suppose that $G$ is a unicyclic graph of length $\ell (\geq 3)$ and  $j\in [1, \lfloor\frac{\ell-1}{2}\rfloor]$. If $G \in \mathcal{O}(\ell)$ satisfies either $j=\lfloor\frac{\ell-1}{2}\rfloor$ or the following two conditions (i) and (ii), then  we call $G$ an \emph{odd-$(\ell, j)$} graph;  if $G \in \mathcal{E}(\ell)$ satisfies the following three conditions (i), (ii), and (iii), then  we call $G$ an \emph{even-$(\ell, j)$} graph.
 \vspace{-0.2cm}
\begin{description}
  \item[(i)] Neither $a_0$ nor $a_j$ is a bad vertex;
  \vspace{-0.2cm}
  \item[(ii)]If $j< \lfloor\frac{\ell-2}{2}\rfloor$, $V(T_{a_k})=\{a_k\}$ for every $k\in \{\ell-1, \ldots, \ell-(\lfloor\frac{\ell-2}{2}\rfloor -j)\}\cup \{j+1,\ldots, j+(\lfloor\frac{\ell-2}{2}\rfloor -j)\}$. 
\vspace{-0.2cm}
  \item[(iii)]  If $j>1$ and $T_{a_k}$ contains no branching vertex for every $k\in [1,j-1]$,  then the length of $T_{a_k}$ is at most $\frac{\ell-2}{2} -j$.
\end{description}
\vspace{-0.2cm}
We use $\mathcal{O}(\ell,j)$ and $\mathcal{E}(\ell,j)$ to denote the set of all odd-$(\ell,j)$ graphs and even-$(\ell,j)$ graphs, respectively. See Figure \ref{example-1} for some examples of odd-$(\ell,j)$ graphs and even-$(\ell,j)$ graphs.
\end{definition}

\begin{figure}[H]
  \centering
  \includegraphics[width=12cm]{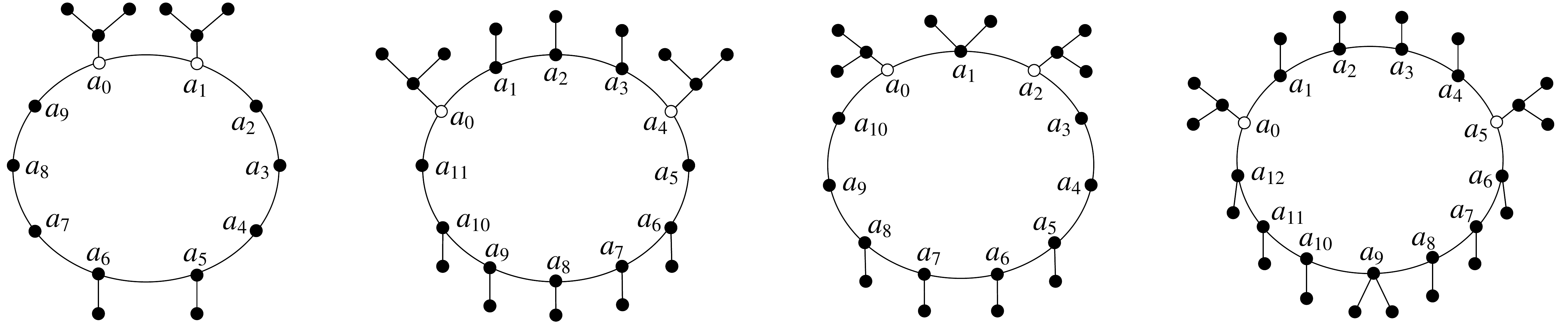}\\
  (a)\hspace{2.7cm} (b)\hspace{2.7cm} (c)\hspace{2.7cm} (d)
  \caption{
  (a) even-(10,1); (b)  even-(12,4); (c)  odd-(11,2); (d)  odd-(13,5)}\label{example-1}
\end{figure}

We observe that  when $\ell \equiv 1$ (mod 2), $\{a_0, a_{\frac{\ell-1}{2}}\}$ is enough to distinguish any pair of vertices that are not in the same component of $G-E(\mathcal{C})$. But this result does not hold for other cases.

\begin{lemma}\label{odd-cycle-1}
Let $G\in \mathcal{O}(\ell)$ and $S\in \mathcal{B}(G)$. If there are two $S$-active vertices  whose distance is $\frac{\ell-1}{2}$, then $S$ is a metric generator of $G$.
\end{lemma}

\begin{proof}
By symmetry, let $a_0$ and $a_{\gamma}$ be two $S$-active vertices,  where $\gamma=\frac{\ell-1}{2}$.
 Let $x,y\in V(G)$ be two distinct vertices, and $a_i$ and $a_j$ ($i,j\in[0,\ell-1]$) be the masters of $x$ and $y$, respectively. By Lemma \ref{Sedlar-1}, let $i\neq j$. Moreover, if $\{i,j\} \cap \{0,\gamma\}\neq \emptyset$, say $i=0$ (the other cases can be discussed in the same way), then  under the condition of $d_{G}(a_0, x)=d_{G}(a_0,y)$ we deduce that
 $d_G(a_{\gamma}, x)= d(a_0,x)+\gamma= d(a_0,y)+\gamma=d_G(a_0, a_j)+d_G(a_j,y)+\gamma > d_G(a_j,y)+d_G(a_j, a_{\gamma})=d_G(a_{\gamma}, y)$ (since $d_G(a_0, a_j)>0$ and $d_G(a_j, a_{\gamma})< \gamma$). Therefore, we assume $i<j$ and $i,j \in [0,\ell-1]\setminus \{0,\gamma\}$. Then, by Lemma \ref{add3}, $x$ and $y$ can be distinguished by $\{a_0,a_{\gamma}\}$, and also by $S$.
\end{proof}

Let us now turn to a more general sufficient condition for a unicyclic graph $G$ satisfying $\mathrm{dim}(G)$=$L(G)$.

\begin{lemma} \label{ok-graph-1}
Let $G$ be a unicyclic graph with $|\mathcal{A}(S)|\geq 2$, where $S\in \mathcal{B}(G)$. Label $\mathcal{C}(G)$ normally. If there exists an integer $j\in [1, \lfloor\frac{\ell-1}{2}\rfloor ]$ such that $a_j\in \mathcal{A}(S)$ and $G\in \mathcal{O}(\ell,j) \cup \mathcal{E}(\ell,j)$, then $\mathrm{dim}(G)$=$L(G)$.
\end{lemma}
\begin{proof}
By Lemma \ref{necessary-0},  it suffices to prove that $S$ is a metric generator. If $\ell\equiv 1$ (mod 2) and $j=\frac{\ell-1}{2}$, then the conclusion follows from Lemma \ref{odd-cycle-1}. So, assume that $1\leq j\leq \lfloor\frac{\ell-2}{2}\rfloor$.
Let $x,y$ be two distinct vertices in $G$, whose masters are $a_{i'}$ and $a_{j'}$ ($i',j' \in [0,\ell-1]$), respectively. By Lemmas \ref{Sedlar-1} and \ref{add1}, suppose that $i'\neq j'$ (say $i'<j'$) and $\{i',j'\}\cap \{0,j\}=\emptyset$. Let $u\in V(T_{a_0})\cap S$, $v \in V(T_{a_j})\cap S$, $m_1= d_G(x, a_{i'})$, and $m_2=d_G(y,a_{j'})$. In the below, we  suppose
\begin{equation}\label{equ2-1}
d_{G}(u, x)=d_{G}(u,y)~~and~~ d_{G}(v, x)=d_{G}(v,y), ~~i.e., d_{G}(a_0, x)=d_{G}(a_0, y) ~~and ~~d_{G}(a_j, x)=d_{G}(a_j, y)
\end{equation}
By Lemma \ref{add3}, one of the following conditions hold: (i) $j'\in [j+1, \lfloor\frac{\ell}{2}\rfloor]$, $i'>j$, and $m_1-m_2=j'-i'$; (ii) $j'\in [\lfloor\frac{\ell}{2}\rfloor+1, \lfloor\frac{\ell}{2}\rfloor+j]$, $i'<j$, $\ell=2i'+2j'-2j$, and $m_1-m_2=\ell-j'-i'$; (iii) $j'\in [\lfloor\frac{\ell}{2}\rfloor+j+1, \ell-1]$, $\ell=2(i'-j)$, and $m_1-m_2=i'-j'$; or  $i'> \lfloor\frac{\ell}{2}\rfloor+j$ and $m_1-m_2=i'-j'$.

For (i), by $G\in \mathcal{O}(\ell,j) \cup \mathcal{E}(\ell,j)$ (Definition \ref{def-new1} (ii)), we deduce that $m_1=0$ and $m_2=i'-j'<0$, a contradiction.
For (ii), we have that $\ell\equiv 0$ (mod 2) and $m_1=\frac{\ell}{2}-j+m_2 \geq \frac{\ell}{2}-j$. Therefore, by  $G\in \mathcal{E}(\ell,j)$ (Definition \ref{def-new1} (iii)),  there exists a vertex $k\in [1,j-1]$ such that $T_{a_k}$ contains a branching vertex. So, $S\cap V(T_{a_k})\neq \emptyset$, say $w\in S\cap V(T_{a_k})$. By Lemma \ref{add2}, $x$ and $y$ can be distinguished by $w$.
For (iii), we have that $j'\geq \lceil\frac{\ell+2}{2}\rceil+j$ and $m_2=j'-i'+m_1>0$ (which implies that $|V(T_{a_{j'}})|>1$), a contradiction to $G\in \mathcal{E}(\ell,j) $ (Definition \ref{def-new1} (ii)).

\end{proof}

Based on Lemma \ref{ok-graph-1}, we can obtain one of our main results (for metric dimension) as follow.

\begin{theorem} \label{two}
Suppose that $G$ is a unicyclic graph with $|\mathcal{A}(S)|\geq 2$, where  $S\in \mathcal{B}(G)$. Label $\mathcal{C}(G)$ normally and let $j = \max \{i: a_i \in \mathcal{A}(S)\}$. Then, $S$ is a metric generator iff there are three $S$-active vertices forming a geodesic triple or  $G\in \mathcal{O}(\ell,j) \cup \mathcal{E}(\ell,j)$.
\end{theorem}

\begin{proof}
The sufficiency follows from Lemmas \ref{geodesic-triple} and \ref{ok-graph-1}. Suppose that $S$ is a metric generator. We also suppose that there are not three $S$-active vertices forming a geodesic triple, by which we derive $j\leq \lfloor\frac{\ell}{2}\rfloor$.
In particular, when $j=\lfloor\frac{\ell}{2}\rfloor$ and $G\in \mathcal{E}(\ell)$, it has that $\mathcal{A}(S)=\{a_0,a_j\}$ and $d_G(s, a_i)=d_G(s, a_{\ell-i})$ for every $i\in[1, \frac{\ell-2}{2}]$ and every $s\in S$, a contradiction. In the below, we assume  that  $j\leq \lfloor\frac{\ell-2}{2}\rfloor$ and we will prove that $G\in \mathcal{O}(\ell,j)\cup \mathcal{E}(\ell,j)$.

Suppose, to the contrary, that $G\notin \mathcal{O}(\ell,j)$ (resp. $G\notin\mathcal{E}(\ell,j)$). Then,  at least one of the following conditions (i) and (ii) (resp. (i), (ii), and (iii)) holds.

 \vspace{0.01cm}

(i) $a_0$ or $a_j$ is a bad vertex;

 \vspace{0.01cm}

(ii) When $ j<\lfloor\frac{\ell-2}{2}\rfloor$, there exists a $k\in (\{\ell-1, \ldots, \ell-(\lfloor\frac{\ell-2}{2}\rfloor-j)\}\cup \{j+1,\ldots, j+(\lfloor\frac{\ell-2}{2}\rfloor-j)$ such that $|V(T_{a_k})|>1$.

 \vspace{0.01cm}

(iii) When $j>1$ and $T_{a_{k'}}$ contains no branching vertex for every $k'\in [1,j-1]$,  there exists some $k\in [1,j-1]$ such that the length of $T_{a_{k}}$ is  at least $\frac{\ell}{2}-j$.

 \vspace{0.01cm}

For (i), suppose that $a_0$  is a bad vertex (the case for $a_j$ can be discussed similarly). By the selection of $S$, there is a thread $T'$ attached to $a_0$ which contains no vertex of $S$. Let $u\in V(T')$ be the vertex adjacent to $a_0$. Then,  $d_G(u, s)=1+j'+d_G(a_{j'},s)=d_G(a_{\ell-1}, s)$ for any $s\in S\cap (\cup_{0\leq j'\leq j} V(T_{a_{j'}}))$. So, $S$ can not distinguish  $u$ and $a_{\ell-1}$, a contradiction.

For (ii), consider the vertex $u\in V(T_{a_k})$ that is adjacent to $a_k$. If $k\in \{\ell-1, \ldots, \ell-(\lfloor \frac{\ell-2}{2}\rfloor-j)\}$, then $k-1\geq \lceil\frac{\ell}{2}\rceil+j$,  and hence  $d_G(u, s)=1+ (\ell-k)+j' + d_G(a_{j'},s)= d_G(a_{k-1}, s)$ for any  $s\in S\cap (\cup_{0\leq j'\leq j} V(T_{a_{j'}}))$. So, $S$ can not distinguish $u$ and $a_{\ell-1}$, a contradiction.  If $k \in\{j+1,\ldots, j+(\lfloor\frac{\ell-2}{2}\rfloor-j)$, then  $k+1\leq \lfloor\frac{\ell}{2}\rfloor$,  and hence $d_G(u, s)=1+ (k-j') + d_G(a_{j'},s)= d_G(a_{k+1}, s)$ for any $s\in S\cap (\cup_{0\leq j'\leq j} V(T_{a_{j'}}))$. So, $S$ can not distinguish $u$ and $a_{\ell-1}$, a contradiction.

For (iii), it is clear that $S$ contains no vertex of $T_{a_{j'}}$ for any $j'\in [1,j-1]$, i.e., $S$ contains only vertices in $V(T_{a_0})\cup V(T_{a_j})$.  Let $u\in V(T_{a_k})$ be the vertex that has distance $\frac{\ell}{2}-j$ with $a_k$.
Consider the vertex $a_{\frac{\ell}{2}+j-k}$. Since $\frac{\ell}{2}+j-k>\frac{\ell}{2}$ and $\frac{\ell}{2}+j-k-j= \frac{\ell}{2}-k <\frac{\ell}{2}$,
we have that $d_G(a_{\frac{\ell}{2}+j-k}, s) = \frac{\ell}{2}-j+k + d(a_0,s) =d_G(u, a_0)$  for any $s \in S\cap V(T_{a_0})$ and   $d_G(a_{\frac{\ell}{2}+j-k}, s')  =(\frac{\ell}{2}-j)+(j-k)+ d_G(a_j,s') = d_G(u, a_j)$  for any $s' \in S\cap V(T_{a_j})$. S, $S$ can not distinguish $u$ and $a_{\frac{\ell}{2}+j-k}$, a contradiction.

\end{proof}


Observe that by Lemma \ref{necessary-0} every metric generator $S'$ is a branch-resolving set with $|\mathcal{A}(S')| \geq 2$. In addition, for every two distinct vertices $a_p, a_q \in V(\mathcal{C})$ we can always find a  vertex $a_r \in V(\mathcal{C}) \setminus \{a_p,a_q\}$ such that $a_p,a_q,a_r$ form a geodesic triple.
 So, 2 $\leq$ dim($G$)$\leq 3$ when $L(G)=0$,  $L(G)+1 \leq$ dim($G$)$\leq L(G)+2$ when $L(G)=1$, and  $L(G) \leq$ dim($G$)$\leq L(G)+1$ when $L(G)\geq 2$. So, the  following result holds by Lemma \ref{two}.

\begin{corollary} \label{th1}
Let $G$ be a unicyclic graph and $S\in \mathcal{B}(G)$. Label $\mathcal{C}(G)$ normally when $|\mathcal{A}(S)|\geq 2$ and let $j = \max \{i: a_i \in \mathcal{A}(S)\}$.  If $G\in \mathcal{O}(\ell)$, then

\vspace{0.1cm}
(1) When $\mathcal{A}(S)=\emptyset$,  $\mathrm{dim}(G)=2;$
\vspace{0.1cm}

(2) When $|\mathcal{A}(S)|=1$, $\mathrm{dim}(G)=L(G)+1;$
\vspace{0.1cm}

(3) When $|\mathcal{A}(S)|\geq 2$, if there are three $S$-active vertices forming a geodesic triple or $G\in \mathcal{O}(\ell,j)$, then $\mathrm{dim}(G)=L(G)$; otherwise, $\mathrm{dim}(G)=L(G)+1$.

Moreover, if $G\in \mathcal{E}(\ell)$, then

\vspace{0.1cm}
(1) When $\mathcal{A}(S)=\emptyset$,  if there are two vertices on $\mathcal{C}$ (say $a_0$ and $a_k$)  such that $G\in \mathcal{E}(\ell,k)$, then dim($G$)=2; otherwise dim($G$)=3.

(2) When $|\mathcal{A}(S)|=1$, let $\mathcal{A}(S)=\{a_{k}\}$.  If there exists a vertex $a_{k'} \in V(\mathcal{C})\setminus \{a_k\}$  such that $G\in \mathcal{E}(\ell,k')$ (let $k=0$) or $G\in \mathcal{E}(\ell,k)$ (let $k'=0$), then  dim(G)=$L(G)+1$; otherwise, dim(G)=$L(G)$+2.

(3) When $|\mathcal{A}(S)|\geq 2$, if there are three $S$-active vertices forming a geodesic triple or  $G\in \mathcal{O}(\ell,j)$, then dim(G)=$L(G)$; otherwise, dim(G)=$L(G)+1$.
\end{corollary}

\section{Edge Metric Dimension}

The previous section has shown that the metric dimension of unicyclic graphs can be exactly determined by the aid of odd- and even-$(\ell, j)$ graphs. This section moves on to consider the edge metric dimension of unicyclic graphs. Before proceeding further, it is important to introduce the other two families of graphs.

\begin{definition}\label{def-new2}
Suppose that $G$ is a unicyclic graph of length  $\ell (\geq 3)$ and $j\in [1, \lfloor\frac{\ell-1}{2}\rfloor]$. If $G \in \mathcal{O}(\ell)$ satisfies the following condition (i), (ii), and (iii), then we call $G$ an \emph{edge-odd-$(\ell,j)$} graph; if $G \in \mathcal{O}(\ell)$ satisfies the following conditions (i), (ii), and (iv), then we call $G$ an \emph{edge-even-$(\ell,j)$} graph.

\vspace{-0.2cm}
\begin{description}
\item[(i)] Neither $a_0$ nor $a_j$ is a bad vertex;
  \vspace{-0.2cm}
\item[(ii)] When $j< \lfloor\frac{\ell-1}{2}\rfloor$, $V(T_{a_k})=\{a_k\}$ for every $k\in \{\ell-1, \ldots, \ell-(\lfloor\frac{\ell-1}{2}\rfloor-j)\}\cup \{j+1,\ldots, j+(\lfloor\frac{\ell-1}{2}\rfloor-j)\}$. 
  \vspace{-0.2cm}
\item[(iii)] When $j >1$ and  $T_{a_k}$ contains no branching vertex for every $k\in [1,j-1]$,  then the length of $T_{a_k}$ is at most $\frac{\ell-1}{2}-j$.
  \vspace{-0.6cm}
  \item[(iv)] When $j >1$ and both $T_{a_k}$ and $T_{a_{k'}}$ contain no branching vertex for every $k\in [1,j-1]$ and $k'=\frac{\ell}{2}+j-k$, if $|V(T_{a_{k'}})|>1$, then the length of  $T_{a_k}$ is at most $\frac{\ell}{2}-j$.
\end{description}
\vspace{-0.2cm}
We use $\mathcal{O}_e(\ell,j)$ and $\mathcal{E}_e(\ell,j)$  to denote the set of all edge-odd-$(\ell,j)$ graphs and edge-even-$(\ell,j)$ graphs, respectively. See Figure \ref{example-2} for some examples of edge-odd-$(\ell,j)$ graphs and edge-even-$(\ell,j)$ graphs.
\end{definition}

\begin{figure}[H]
  \centering
  \includegraphics[width=12cm]{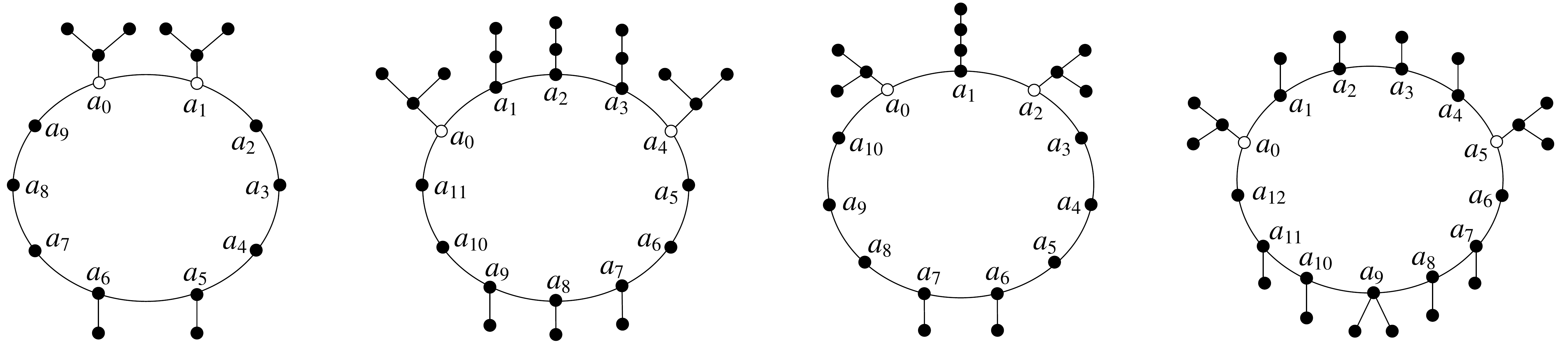}\\
  (a)\hspace{2.7cm} (b)\hspace{2.7cm} (c)\hspace{2.7cm} (d)
  \caption{
  (a) Edge-even-(10,1); (b)  edge-even-(12,4); (c)  edge-odd-(11,2); (d)  edge-odd-(13,5)}\label{example-2}
\end{figure}

Let us now turn to the discussion of the edge metric dimension of unicyclic graphs based on edge-odd- and edge-even-graphs.
Although the conclusions (and their proofs) in this section  are analogous as those in Section \ref{sec-1}, there are still a number of important differences between them. So, we fully describe them  as well.

\begin{lemma} \label{ok-graph-1-0}
Let $G$ be a unicyclic graph with $|\mathcal{A}(S)|\geq 2$, where $S\in \mathcal{B}(G)$. Label $\mathcal{C}(G)$ normally. If there exists a $j\in [1, \lfloor\frac{\ell-1}{2}\rfloor]$ such that $a_j\in \mathcal{A}(S)$ and $G\in \mathcal{O}_e(\ell,j)\cup \mathcal{E}_e(\ell,j)$, then $\mathrm{edim}(G)$=$L(G)$.
\end{lemma}
\begin{proof}
By Lemma \ref{necessary-0}, it suffices to show that $S$ is an edge metric generator.  Observe that $\mathcal{C}$ is labelled normally. If there exists $r=\min\{i|i>\lfloor\frac{\ell-1}{2}\rfloor$ and $a_i$ is  $S$-active\},  then $a_0,a_j,a_r$ form a geodesic triple and $S$ is an edge metric generator of $G$ by Lemma \ref{geodesic-triple}. Therefore, we assume that $T_{a_{r}}$ contains no branching vertex  for every $\lfloor\frac{\ell-1}{2}\rfloor < r \leq \ell-1$.

Let $e_1,e_2$ be an arbitrary pair of edges in $G$.
Clearly, if both $e_1\in E(\mathcal{C})$ and $e_2\in E(\mathcal{C})$, then $e_1$ and $e_2$ can be distinguished by by $\{a_0, a_j\}$, and also by $S$.  By Lemmas \ref{Sedlar-1} and \ref{add1}, $e_1$  and $e_2$ can be distinguished by $S$ if $e_1$ and $e_2$ belong to the same component of $G-E(\mathcal{C})$ or $\{e_1,e_2\}\cap (E(T_{a_0})\cup E(T_{a_j}))\neq \emptyset$.
In the below, we assume that  $e_2\in E(T_{a_{j'}})$ for some $j'\in [0, \ell-1]$, $\{e_1,e_2\}\cap (E(T_{a_0})\cup E(T_{a_j}))=\emptyset$ (which implies that $j'\notin \{0,j\}$), and there does not exist any $i\in [0, \ell-1]$ such that $\{e_1,e_2\}\subseteq E(T_{a_i})$.  Let $e_2=w_1w_2$ and $d_G(a_{j'}, e_2)=d_G(a_{j'}, w_1)=m_2$. This also implies that
$d_G(a_0,e_2)=d_G(a_0, w_1)$ and $d_G(a_j,e_2)=d_G(a_j, w_1)$.
Now, suppose that
\begin{equation}\label{equ-2}
d_G(a_0, e_1) = d_G(a_{0}, e_2)~~ \mathrm{and}~~d_G(a_j, e_1) = d_G(a_{j}. e_2)
\end{equation}
We will derive a contradiction or show that $e_1$ and $e_2$ can be distinguished by a vertex in $S\setminus (V(T_{a_0}) \cup V(T_{a_j}))$.

Case 1.  $e_1\in E(\mathcal{C})$.  Let $e_1=a_{i'}a_{i'+1}$, where $i'\in [0,\ell-1]$ and $a_{\ell}=a_0$. By $G\in \mathcal{O}_e(\ell,j)\cup \mathcal{E}_e(\ell,j)$ (Definition \ref{def-new2} (ii)),
 we have that $\lfloor\frac{\ell-1}{2}\rfloor<j'<\ell-\lfloor\frac{\ell-1}{2}\rfloor+j$ or $j' \in [1,j-1]$ (Note that $j'\notin \{0,j\}$). Moreover, when $j'>\lfloor\frac{\ell-1}{2}\rfloor$, it follows that $i'<j'$; otherwise, $d_G(a_0, e_1)=d_G(a_0, a_{i'+1}) < d_G(a_0, a_{j'}) \leq d_G(a_0, e_2)$.

Case 1.1. $j' \in [1,j-1]$.  Then, $d_G(a_0, e_2)=j'+ m_2$ and $d_G(a_j, e_2)=j-j'+ m_2$. If $i'\leq j$, then either $d_G(a_0, e_1) < d_G(a_0, e_2)$ (when $i'<j'$) or $d_G(a_j, e_1) < d_G(a_j, e_2)$ (when $i'\geq j'$), a contradiction to Equation (\ref{equ-2}).
If $j< i'\leq \lfloor\frac{\ell-1}{2}\rfloor$,  then $d_G(a_0, e_1)=i'$ and $d_G(a_j, e_1)=i'-j$. By Equation (\ref{equ-2}), we deduce that $j=j'$, a contradiction.
If $\lfloor\frac{\ell-1}{2}\rfloor+j <i'\leq \ell-1$, then  $d_G(a_0, e_1)=\ell-(i'+1)$ and $d_G(a_j, e_1)=\ell-(i'+1)$ +$j$ (Notice that in this case $i'$ does not exist when $\ell\equiv 1$ (mod 2) and $j=\frac{\ell-1}{2}$). By Equation (\ref{equ-2}), we deduce that $j'=0$, a contradiction.

Moving on now to consider $\lfloor\frac{\ell-1}{2}\rfloor< i'\leq \lfloor\frac{\ell-1}{2}\rfloor+j$, it follows that $d_G(a_0, e_1)=\ell-i'-1$ and $d_G(a_j, e_1)=i'-j$. Then, by Equation (\ref{equ-2}), we derive  $\ell=2i'+2j'-2j+1$ and $m_2= i'+j'-2j=\frac{\ell-1}{2}-j$.
This shows that there does exist a $j'' \in [1, j-1]$ for which $S\cap V(T_{a_{j''}})\neq \emptyset$, since if not, $T_{a_r}$ contains no branching vertex for every $r\in [1,j-1]$ while $T_{a_{j'}}$ has length at least $m_2+1 =\frac{\ell-1}{2}-j +1 $,  a contradiction to
$G\in \mathcal{O}_e(\ell,j)$ (Definition \ref{def-new2} (iii)).
Let $w\in S\cap V(T_{a_{j''}})\neq \emptyset$ and we will prove that
$e_1$ and $e_2$ can be distinguished by $w$.

For this, we construct a new graph $G'$ that is obtained from $G$ by subdividing  edge $e_1$ (i.e., delete $e_1$, add a new vertex, say $a'$, and join $a'$ to $a_{i'}$ and $a_{i'+1}$). Let $\ell'=\ell+1$ and consider $G'$. Clearly, $G'\in \mathcal{E}(\ell')$, $j\in [1, \frac{\ell'-2}{2}]$, and $S\in \mathcal{B}(G')$.
As regards  $\mathcal{C}(G')$,  a normal labeling  can be obtained based on  $\mathcal{C}(G)$ by relabelling $a'$ with $a_{i'+1}$ and  $a_{i}$ with $a_{i+1}$ for every  $i\in \{i'+1, \ldots, \ell-1\}$.
Notices that $j', j'' \in [1,j-1]$, $\frac{\ell'}{2}+1\leq  i'+1\leq \frac{\ell'}{2}+j$,  $d_G(a_0,w_2)=d_G(a_0,e_2)+1$, and  $d_G(a_j,w_2)=d_G(a_j,e_2)+1$. So,
$d_{G'}(a_0, a_{i'+1})=d_G(a_0, a_{i'})=d_G(a_0,e_1)+1=d_G(a_0,e_2)+1=d_G(a_0, w_2)$ and $d_{G'}(a_j, a_{i'+1})=d_G(a_j, a_{i'})+1 = d_G(a_j,e_1)+1=d_G(a_j,e_2)+1=d_G(a_j, w_2)$. Then, by Lemma \ref{add2}, $d_{G'} (w, a_{i'+1}) \neq d_{G'} (w, w_2)$. For graph $G$ (and also for $G'$), if $e_2$ is on $P_{w, a_{j'}}$, then clearly $d_G(w, e_2)<d_G(w,e_1)$. We therefore assume that $e_2$ is not on $P_{w, a_{j'}}$, which implies that $d_G(w, e_2) = d_G(w, w_1) =d_G(w, w_2)-1= d_{G'}(w,w_2)-1$. Additionally, when (in $G'$) $i'+1 - j'' \leq \frac{\ell'}{2}$, it holds that (in $G$) $i' - j'' \leq \frac{\ell-1}{2}$ and hence $d_G(w,e_1)= d_G(w, a_{i'}) = d_{G'}(w, a_{i'+1})-1$; when  (in $G'$) $i'+1 - j'' > \frac{\ell'}{2}$, it holds that (in $G$) $i' - j'' > \frac{\ell-1}{2}$ and hence $d_G(w,e_1)= d_G(w, a_{i'+1}) = d_{G'}(w, a_{i'+2})= d_{G'}(w, a_{i'+1})-1$. As a conclusion, we have that $d_G(w,e_1)=d_{G'}(w, a_{i'+1})-1 \neq d_{G'}(w, w_2)-1 = d_G(w,e_2)$.

Case 1.2. $\lfloor\frac{\ell-1}{2}\rfloor<j'<\ell-\lfloor\frac{\ell-1}{2}\rfloor+j$. Then, $i'<j' \leq  \ell-\lfloor\frac{\ell-1}{2}\rfloor+j-1$, and $d_G(a_0, e_2) = \ell-j' + m_2$ and $d_G(a_j, e_2) = j' -j + m_2$.
Furthermore, if $0\leq i' \leq  \lfloor\frac{\ell-1}{2}\rfloor$, then $d_G(a_0, e_1)=i'$, and $d_G(a_{j}, e_1)=i'-j$ (when $i'\geq j$) or  $d_G(a_{j}, e_1)=j-(i'+1)$ (when $i'<j$). By Equation (\ref{equ-2}), we derive $m_2=i'-j'<0$ (when $i'\geq j$) or $m_2= j-\frac{\ell+1}{2}<0$ (when $i'<j$), a contradiction.
If $\lfloor\frac{\ell-1}{2}\rfloor+1\leq i' <\ell-\lfloor\frac{\ell-1}{2}\rfloor+j-1$, then $i'-j<\lfloor\frac{\ell}{2}\rfloor$, and $d_G(a_0, e_1)=\ell- (i'+1)$ and  $d_G(a_{j}, e_1)=i'-j$. By Equation (\ref{equ-2}), we have that $2i'=2j'-1$, a contradiction.

Case 2. $e_1\notin E(\mathcal{C})$. Let $e_1\in T_{a_{i'}}$ for some $i'\in [0, \ell-1]\setminus \{0,j,j'\}$ and $d_G(a_{i'}, e_1)=m_1$. Without loss of generality, we assume that $i'<j'$. Then, by Equation (\ref{equ-2}) and Lemma \ref{add3}, one of the following conditions hold: (i) $j'\in [j+1, \lfloor\frac{\ell}{2}\rfloor]$, $i'>j$, and $m_1-m_2=j'-i'$; (ii) $j'\in [\lfloor\frac{\ell}{2}\rfloor+1, \lfloor\frac{\ell}{2}\rfloor+j]$, $i'<j$, $\ell=2i'+2j'-2j$, and $m_1-m_2=\ell-j'-i'$; (iii) $j'\in [\lfloor\frac{\ell}{2}\rfloor+j+1, \ell-1]$, $\ell=2(i'-j)$, and $m_1-m_2=i'-j'$; or  $i'> \lfloor\frac{\ell}{2}\rfloor+j$ and $m_1-m_2=i'-j'$.

For (i), by $G\in \mathcal{O}_e(\ell,j) \cup \mathcal{E}_e(\ell,j)$ (Definition \ref{def-new2} (ii)), we deduce that $m_1=0$ and $m_2=i'-j'<0$, a contradiction.
For (ii), we have that $\ell\equiv 0$ (mod 2), $j'=\frac{\ell}{2}+j-i'$, and $m_1=\frac{\ell}{2}-j+m_2 \geq \frac{\ell}{2}-j$ (which implies that the length of $T_{a_{i'}}$ is at least $m_1+1\geq \frac{\ell}{2}-j+1$).
 Therefore, by  $G\in \mathcal{E}_e(\ell,j)$ (Definition \ref{def-new2} (iv)),  there exists a vertex $k\in [1,j-1]$ such that $T_{a_k}$ contains a branching vertex. So, $S\cap V(T_{a_k})\neq \emptyset$, say $w\in S\cap V(T_{a_k})$. By Lemma \ref{add2}, $x$ and $y$ can be distinguished by $w$.
For (iii), we have that $j'\geq \lceil\frac{\ell+2}{2}\rceil+j$ and $|V(T_{a_{j'}})|>1$, a contradiction to $G\in \mathcal{E}(\ell,j) $ (Definition \ref{def-new2} (ii)).
\end{proof}

\begin{theorem} \label{two-2-1}
Suppose that $G$ is a unicyclic graph with $|\mathcal{A}(S)|\geq 2$, where  $S\in \mathcal{B}(G)$. Label $\mathcal{C}(G)$ normally and let $j = \max \{i: a_i \in \mathcal{A}(S)\}$. Then, $S$ is a metric generator iff there are three $S$-active vertices forming a geodesic triple or $G\in \mathcal{O}_e(\ell,j) \cup \mathcal{E}_e(\ell,j)$.
\end{theorem}

\begin{proof}
The sufficiency follows from Lemmas \ref{geodesic-triple}, and \ref{ok-graph-1-0}. Suppose that $S$ is an edge metric generator, and there are not three $S$-active vertices forming a geodesic triple.  This shows that $j\leq \lfloor\frac{\ell}{2}\rfloor$, and  $\mathcal{A}(S)=\{a_0,a_j\}$ when $j=\lfloor\frac{\ell}{2}\rfloor$ and $G\in \mathcal{E}(\ell)$ (however, in this case $d_G(s, a_ia_{i+1})=d_G(s, a_{\ell-i-1}a_{\ell-i})$ for every $i\in[0, \frac{\ell-2}{2}]$ and every $s\in S$, where $a_{\ell}=a_0$, a contradiction).
In the below, we assume  that  $j\leq \lfloor\frac{\ell-1}{2}\rfloor$ and we will prove that $G\in \mathcal{O}(\ell,j)\cup \mathcal{E}(\ell,j)$. To the contrary, if $G\notin \mathcal{O}_e(0,j)$ (resp. $G\notin\mathcal{E}_e(0,j)$), then at least one of the following conditions (i), (ii), and (iii) (resp. (i), (ii), and (iv)) holds.

(i) $a_0$ or $a_j$ is a bad vertex;

 \vspace{0.01cm}

(ii) When $j<\lfloor\frac{\ell-1}{2}\rfloor$, there exists a $k\in (\{\ell-1, \ldots, \ell-(\lfloor\frac{\ell-1}{2}\rfloor-j)\}\cup \{j+1,\ldots, j+(\lfloor\frac{\ell-1}{2}\rfloor-j)$ such that $|V(T_{a_k})|>1$.

\vspace{0.01cm}

(iii) When $j>1$ and $T_{a_{k'}}$ contains no branching vertex for every $k'\in [1,j-1]$,  there exists a $k\in [1,j-1]$ such that $T_{a_{k}}$ is  a path of length at least $\frac{\ell+1}{2}-j$.

\vspace{0.01cm}

(iv) When $j>1$ and both $T_{a_{i}}$ and $T_{a_{i'}}$ contain no branching vertex for every $i\in [1,j-1]$ and $i'=\frac{\ell}{2}+j-i$, there exists a $k\in [1,j-1]$ and a $k'= \frac{\ell}{2}+j-k$ such that $|V(T_{a_{k'}})|>1$ and
$T_{a_{k}}$ has length at least $\frac{\ell}{2}-j+1$.

For (i), suppose $a_0$ is a bad vertex. By the selection of $S$, there is a thread $T'$ attached to $a_0$ which contains no vertex of $S$. Let $u\in V(T')$ be the vertex adjacent to $a_0$. Then,  $d_G(s,ua_0)=j'+d_G(s, a_{j'})=d_G(s, a_{\ell-1}a_0)$ for any  $s\in S\cap (\cup_{0\leq j'\leq j} V(T_{a_{j'}}))$. Therefore, $S$ can not distinguish  $ua_0$ and $a_{\ell-1}a_0$.

For (ii), consider the vertex $u\in V(T_{a_k})$ that is adjacent to $a_k$. If $k\in \{\ell-1, \ldots, \ell-(\lfloor\frac{\ell-1}{2}\rfloor-j)\}$, then $k-1\geq \lceil\frac{\ell-1}{2}\rceil+j$,  and hence $d(s, ua_k)=(\ell-k) + j'+ d_G(s,a_0)= d(s, a_{k-1}a_k)$ for any  $s\in S\cap (\cup_{0\leq j'\leq j} V(T_{a_{j'}}))$. Therefore, $S$ can not distinguish $ua_k$ and $a_{k-1}a_k$. If $k \in\{j+1,\ldots, j+(\lfloor\frac{\ell-1}{2}\rfloor-j)$, then  $d_G(s, ua_k)= d_G(s, a_k) = d_(s, a_ka_{k+1})$ for any  $s\in S\cap (\cup_{0\leq j'\leq j} V(T_{a_{j'}}))$. Therefore, $S$ can not distinguish $ua_k$ and $a_ka_{k+1}$.

For (iii), $\ell \equiv 1$ (mod 2) and $S$ contains only vertices in $V(T_{a_0})\cup V(T_{a_j})$.  Let $uu'\in E(T_{a_k})$ be the edge that has distance $\frac{\ell-1}{2}-j$ with $a_k$. Consider edge $a_{k'}a_{k'+1}$ where $k'=\frac{\ell-1}{2}+j-k$.
Observe that $k\in [1,j-1]$; we have  $k'=\frac{\ell-1}{2}-k+j >\frac{\ell-1}{2}$ and $k'-j<\frac{\ell-1}{2}$. Therefore, $d_G(s, a_{k'}a_{k'+1}) = \ell-(k'+1)+ d_G(a_0,s) = \frac{\ell-1}{2}-j+k+d_G(a_0,s)= d_G(s, uu')$  for any $s \in S\cap V(T_{a_0})$ and   $d_G(s', a_{k'}a_{k'+1} ) = k'-j + d_G(a_j, s') =(\frac{\ell-1}{2}-j)+(j-k)+ d_G(a_j, s')=d_G(s', uu')$  for any $s' \in S\cap V(T_{a_j})$. Therefore, $S$ can not distinguish $uu'$ and $a_{k'}a_{k'+1}$.

For (iv),  $\ell \equiv 0$ (mod 2) and $S$ contains only vertices in $V(T_{a_0})\cup V(T_{a_j})$.  Let $uu'\in E(T_{a_k})$ be the edge that has distance exactly $\frac{\ell}{2}-j$ with $a_k$, and $va_{k'}\in E(T_{a_{k'}})$, where $k'=\frac{\ell}{2}+j-k$.
Observe that $k\in [1,j-1]$. We have that $k'>\frac{\ell}{2}$ and $k'-j<\frac{\ell}{2}$. Therefore, $d_G(s, va_{k'}) = \ell-k'+ d_G(a_0,s) = (\frac{\ell}{2}-j)+k+d_G(a_0,s)= d_G(s, uu')$  for any $s \in S\cap V(T_{a_0})$ and   $d_G(s', va_{k'}) = k'-j + d_G(a_j, s') =(\frac{\ell}{2}-j)+(j-k)+ d_G(a_j, s')=d_G(s', uu')$  for any $s' \in S\cap V(T_{a_j})$. Therefore, $S$ can not distinguish $uu'$ and $va_{k'}$.
\end{proof}


By Theorem \ref{two-2-1} and with an analogous argument as that for Corollary \ref{th1},  we obtain the following result.
\begin{corollary} \label{th2}
Let $G$ be a unicyclic graph and $S\in \mathcal{B}(G)$. Label $\mathcal{C}(G)$ normally when $|\mathcal{A}(S)|\geq 2$ and let $j = \max \{i: a_i \in \mathcal{A}(S)\}$.  
Then,

\vspace{0.1cm}
(1)  When $\mathcal{A}(S)=\emptyset$,  if there are two vertices on $\mathcal{C}$, say  $a_{0}$ and $a_{j}$, such that $G\in \mathcal{O}_e(\ell,j)\cup  \mathcal{E}_e(\ell,j)$, then edim($G$)=2; otherwise edim($G$)=3.
\vspace{0.1cm}

(2) When $|\mathcal{A}(S)|=1$, let $\mathcal{A}(S)=\{a_k\}$. If there exists a vertex $a_{k'}\in V(\mathcal{C})\setminus \{a_k\}$  such that $G\in \mathcal{O}_e(\ell,k')\cup \mathcal{E}_e(\ell,k')$ (let $k=0$) or $G\in \mathcal{O}_e(\ell,k)\cup \mathcal{E}_e(\ell,k)$ (let $k'=0$), then  edim(G)=$L(G)+1$; otherwise, edim(G)=$L(G)$+2.
\vspace{0.1cm}

(3) When $|\mathcal{A}(S)|\geq 2$,  if there are three $S$-active vertices forming a geodesic triple or  $G\in \mathcal{O}_e(0,j)\cup  \mathcal{E}_e(0,j)$, then edim(G)=$L(G)$; otherwise, edim(G)=$L(G)+1$.

\end{corollary}

\section{Conclusion}

According to Definitions \ref{def-new1} and \ref{def-new2}, we see that $\mathcal{O}_e(\ell,j) \subset \mathcal{O}(\ell,j)$ and  $\mathcal{E}(\ell,j) \subset \mathcal{E}_e(\ell,j)$.
Therefore, by Corollaries \ref{th1} and \ref{th2}, we can determine the relation between $\mathrm{dim}(G)$ and  $\mathrm{edim}(G)$ for unicyclic graphs $G$ (see Table 1), and give an answer to Problems \ref{pro1} and $\ref{pro2}$.   Indeed, by Table 1, we can generate an efficient algorithm to compute the value of $\mathrm{dim}(G)$, $\mathrm{edim}(G)$, and $\mathrm{dim}(G)-\mathrm{edim}(G)$. In addition,
our method can also be used to discuss  metric and edge metric dimensions for cactus graphs, which will be examined in our future work.

\begin{table}\label{tab-1} \caption{The relation between $\mathrm{dim}(G)$ and  $\mathrm{edim}(G)$ for unicyclic graphs $G$.
When $\mathcal{A}(S)$ does not contain three vertices that form a geodesic triple, we make the following assumptions: if $|\mathcal{A}(S)|\geq 2$, $\mathcal{C}$ is labelled normally (which means that $a_0\in \mathcal{A}(S)$) and let $j=\max\{i| a_i$ is $S$-active\}; if $|\mathcal{A}(S)|=1$, let $\mathcal{A}(S)=\{a_k\}, k\in [0,\ell-1]$.  Also, we denote by $R_1:$ $\mathcal{C}$ contains two vertices, say $a_0$ and $a_k$, such that $G\in \mathcal{O}_e(\ell,k)$; $R_2:$ there is a $a_{k'}\in V(\mathcal{C})\setminus \{a_k\}$ such that
 $G\in \mathcal{O}_e(\ell,k')$ (when let $k=0$) or  $G\in \mathcal{O}_e(\ell,k)$ (when let $k'=0$);
   $R_3:$  $G\in \mathcal{O}_e(\ell,j)$ or $G\notin \mathcal{O}(\ell,j)$;  $R_4:$  there are two vertices on $\mathcal{C}$ (say $a_0$ and $a_k$) such that $G\in \mathcal{E}(\ell,k)$, or for every two vertices (also say  $a_0$ and $a_k$) on $\mathcal{C}$, $G\notin \mathcal{E}_e(\ell,k)$;  $R_5:$  there exists a $a_{k'}\in V(\mathcal{C})\setminus \{a_k\}$ such that $G\in \mathcal{E}(\ell,k)$ when let $k'=0$ (or $G\in \mathcal{E}(\ell,k')$ when let $k=0$),  or for every $a_{k'}\in V(\mathcal{C})\setminus \{a_k\}$) $G\notin \mathcal{E}_e(\ell,k)$ (when let $k'=0$) and $G\notin \mathcal{E}_e(\ell,k')$ (when let $k'=0$);
 $R_6:$  $G\in \mathcal{E}(\ell,j)$ or  $G\notin \mathcal{E}_e(\ell,j)$;
 and $\neg R_i$ denotes the negative  of $R_i$ for $i\in [1,6]$.
 }
\begin{center}
\small{
\begin{tabular}{|*{7}{c|}}
\hline
\multicolumn{4}{|m{5cm}<{\centering}|}{\shortstack{\\ A unicyclic graph $G$ and $S\in \mathcal{B}(G)$}} & $\mathrm{dim}(G)$&
$\mathrm{edim}(G)$ & $\mathrm{dim}(G)$-$\mathrm{edim}(G)$\\
\hline
\multicolumn{4}{|m{5cm}<{\centering}|}{\shortstack{ \\$\mathcal{A}(S)$ contains three vertices\\ forming a geodesic triple}}  & $L(G)$&
$L(G)$ & $0$\\
\hline
\multirow{12}{*}{\shortstack{ $\mathcal{A}(S)$\\ contains\\ no \\three \\ vertices\\ forming\\ a \\ geodesic\\ triple}} & \multirow{6}{*}{ \shortstack{$\ell\equiv 1$ \\ (mod 2) }} &  \multirow{2}{*}{ $|\mathcal{A}(S)|=0$ } & {$R_1$}& 2& 2& 0\\
\cline{4-7}
 & &   & $\neg R_1$& 2& 3& -1\\
 \cline{3-7}
 &  &  \multirow{2}{*}{ $|\mathcal{A}(S)|=1$ } & $R_2$& $L(G)+1$& $L(G)+1$& 0\\
  \cline{4-7}
 &  &   & $\neg R_2$& $L(G)+1$& $L(G)+2$& -1\\

  \cline{3-7}
 &  &  \multirow{2}{*}{ $|\mathcal{A}(S)|=2$ } & $R_3$ & $L(G)$ (or L(G)+1)& $L(G)$ (or $L(G)$+1)& 0\\
  \cline{4-7}
 &  &   & $\neg R_3$  & $L(G)$ & $L(G)$+1& -1\\

  \cline{2-7}
  & \multirow{6}{*}{ \shortstack{ $\ell\equiv 0$\\ (mod 2) }} &  \multirow{2}{*}{ $|\mathcal{A}(S)|=0$ } & $R_4$& 2 (or 3)& 2 (or 3)& 0\\
\cline{4-7}
 & &   & $\neg R_4$& 3& 2& 1\\
 \cline{3-7}
 &  &  \multirow{2}{*}{ $|\mathcal{A}(S)|=1$ } & $R_5$& $L(G)+1$ (or $L(G)$+2)& $L(G)+1$ (or $L(G)$+2) & 0\\
  \cline{4-7}
 &  &   & $\neg R_5$& $L(G)+2$& $L(G)+1$& 1\\

  \cline{3-7}
 &  &  \multirow{2}{*}{ $|\mathcal{A}(S)|=2$ } & $R_6$ & $L(G)$ (or L(G)+1)& $L(G)$ (or $L(G)$+1)& 0\\
  \cline{4-7}
 &  &   & $\neg R_6$ & $L(G)+1$ & $L(G)$& 1\\

\hline

\end{tabular}
}
\end{center}
\end{table}

\section{Acknowledgements}

This work was supported in part by  National Natural Science Foundation of China  under Grant 61872101; in part by Natural Science Foundation of Guangdong Province of China under Grant 2021A1515011940.



 \bibliographystyle{elsarticle-num}
  \bibliography{mybib}





\end{document}